\documentclass[12pt]{amsart}
\usepackage[top=1in,left=1in,bottom=1in,right=1in]{geometry}
\usepackage{amsmath}
\usepackage{amssymb}
\usepackage{amsthm}
\usepackage{tikz}
\usetikzlibrary{arrows}
\usepackage{graphicx}
\usepackage{epstopdf}
\usepackage{url}
\usepackage{colordvi}
\usepackage{color}
\usepackage{verbatim}

\newcommand{\R}{{\mathbb R}}

\newcommand{\od}{:=}

\newcommand{\U}{\mathcal{U}}

\newcommand{\C}{\mathcal{C}}
\newcommand{\D}{\mathcal{D}}
\newcommand{\F}{\mathbb{F}}

\newcommand{\CF}{\mathcal{CF}}

\newcommand{\V}{\mathcal V}

\newtheorem{theorem}{Theorem}[section]
\newtheorem{lemma}[theorem]{Lemma}

\newtheorem{corollary}[theorem]{Corollary}
\theoremstyle{definition}
\newtheorem{definition}[theorem]{Definition}
\theoremstyle{remark}

\newtheorem{example}[theorem]{Example}

\DeclareMathOperator{\rest}{rest}
\DeclareMathOperator{\supp}{supp}

\begin{document}
\title{Neural Ideal Preserving Homomorphisms}
%

\author{R. Amzi Jeffs}
\address{Department of Mathematics.  Harvey Mudd College, Claremont, CA 91711}
\email{rjeffs@g.hmc.edu}

\author{Mohamed Omar}
\address{Department of Mathematics.  Harvey Mudd College, Claremont, CA 91711}
\email{omar@g.hmc.edu}

\author{Nora Youngs}
\address{Department of Mathematics.  Colby College, Waterville, ME}
\email{nora.youngs@colby.edu}



\begin{abstract}
The neural ideal of a binary code $\C\subseteq \F_2^n$ is an ideal in $\F_2[x_1,\ldots, x_n]$ closely related to the vanishing ideal of $\C$. 
The neural ideal, first introduced by Curto et al, provides an algebraic way to extract geometric properties of realizations of binary codes. In this paper we investigate homomorphisms between polynomial rings $\F_2[x_1,\ldots, x_n]$ which preserve all neural ideals. We show that all such homomorphisms can be decomposed into a composition of three basic types of maps. Using this decomposition, we can interpret how these homomorphisms act on the underlying binary codes. We can also determine their effect on geometric realizations of these codes using sets in $\R^d$. We also describe how these homomorphisms affect a canonical generating set for neural ideals, yielding an efficient method for computing these generators in some cases.

\end{abstract}

\thanks{}
\date{\today}
\maketitle

\section{Introduction}
\label{sec:intro}

Given a binary code $\C\subseteq \F_2^n$, a \emph{realization} of $\C$ is a collection of sets $\mathcal U = \{U_1,\ldots, U_n\}$ in $\R^d$ such that $\C$ records intersection patterns of the sets in $\mathcal U$. More precisely, an element $c$ is in $\C$ exactly when there is a point in $\R^d$ lying in only the $U_i$'s with $c_i = 1$. Every code has a realization in $\R^d$ for $d\ge 1$ if we place no restrictions on the sets in $\mathcal U$. However, if we require that the sets in $\mathcal U$ satisfy certain geometric properties such as connectedness or convexity not every code has a realization. Such restrictions naturally arise in neuroscientific contexts, as addressed in \cite{neuralring13}.

The question of when a code is realizable is difficult to answer in these restricted settings; however, algebraic techniques have been succesfully used to gain insight into obstructions to codes  being realizable.
Given a code $\C$, the authors in \cite{neuralring13} define the  {\it neural ideal} $J_\C\subseteq \F_2[x_1,\ldots, x_n]$ and the associated {\it neural ring}, which algebraically process and store the information in a neural code. Each neural ideal $J_\C$ has a particular set of generators, its \emph{canonical form}, which reveals the fundamental relationships between sets in any realization of $\C$.

In recent work \cite{mapsbetweencodes}, the authors use the ideal-variety correspondence to relate homomorphisms between neural rings to maps between their associated codes, and focus particularly on those maps between codes which have a natural interpretation in neuroscience. We examine ring homomorphisms from $\F_2[x_1,...,x_n]$ to $\F_2[x_1,...,x_m]$ which {\it preserve} neural ideals in the sense that the image of any neural ideal is again a neural ideal. In Theorem \ref{thm:neuralhomomorphism} we completely classify these homomorphisms, and give a method of decomposing any such homomorphism into a composition of three basic types. Theorem \ref{thm:codeeffects} gives a precise description of how these homomorphisms act on the underlying codes of neural ideals. Theorem \ref{thm:geometry} describes how these maps transform realizations of the underlying codes. 

The structure of the paper is as follows. In Section \ref{sec:prelim}, we provide pertinent definitions and describe in detail the main results of our work.  Sections \ref{sec:algebraproof} and  \ref{sec:codeproof}  are devoted to proving Theorems \ref{thm:neuralhomomorphism} and \ref{thm:codeeffects} respectively. In Section \ref{sec:geometryproof} we prove Theorem \ref{thm:geometry}.   In Section \ref{sec:CFproof} we build on Theorem \ref{thm:neuralhomomorphism} by describing how homomorphisms preserving neural ideals affect the canonical form of neural ideals.

\section{Preliminaries and Main Results}
\label{sec:prelim}
We start by presenting the fundamental definitions for the theory we will be exploring, beginning with the formal definition of a neural code. \begin{definition}\label{def:code}
A \emph{code} or \emph{neural code} is a set $\C\subseteq \F_2^n$ of binary vectors. The vectors in $\C$ are called \emph{codewords}. 
\end{definition}

Throughout we will let $[n]$ denote the set $\{1,2,\ldots , n\}$. For any vector $v\in \F_2^n$ the \emph{support} of $v$ is the set $\supp(v) := \{i\in [n] \mid v_i = 1\}$. 

\begin{definition}\label{def:codeofsets}
Let $\U = \{U_1,\ldots, U_n\}$ be a collection of sets in a space $X\subseteq \R^d$. The \emph{code of $\U$} is the code defined by \[
\C(\U) \od \left\{ v\in \F_2^n \middle| \left(\bigcap_{v_i = 1} U_i\right) \scalebox{2}{$ \setminus$} \left(\bigcup_{v_j = 0} U_j\right) \neq \emptyset  \right\},
\]
where we adopt the convention that the empty intersection is $X$ and the empty union is $\emptyset$.
Given $\C\subset\F_2^n$, we say $\C$ is \emph{realizable} if there exists a collection $\U$ such that $\C = \C(\U)$, and call $\U$ a \emph{realization} of the code in the space $X$. If a code $\C$ has a realization consisting of convex open sets  then we say $\C$ is a \emph{convex code}.
\end{definition}

Classifying which codes are convex is an open problem which has been considered by many researchers \cite{chadvlad, neuralring13, tancer} .  Many partial results exist along with tools for approaching this task. In this paper we build on algebraic tools introduced in \cite{neuralring13} that have been pivotal in understanding obstructions to
convex realizability of codes; see for example \cite{chadvlad,local15,obstructions}. As a code $\C$ is a subset of $\F_2^n$, we will work over the polynomial ring $\F_2[x_1,\ldots, x_n]$;  for convenience we will use $\F_2[n]$ to denote $\F_2[x_1,\ldots, x_n]$. 

\begin{definition}[\cite{neuralring13}]\label{def:pseudomonomial}
A \emph{pseudomonomial} is a polynomial $f\in \F_2[n]$ of the form\[
f = \prod_{i\in\sigma} x_i \prod_{j\in\tau} (1-x_j).
\]
where $\sigma,\tau\subseteq[n]$ and $\sigma\cap \tau=\emptyset$. 
\end{definition}

Note that every pseudomonomial has degree at most $n$. For any vector $v\in \F_2^n$, its \emph{indicator pseudomonomial}, denoted  $\rho_v$, is the degree $n$ pseudomonomial with $\sigma = \supp(v)$ and $\tau = [n]\setminus \supp(v)$.
 Using indicator pseudomonomials, the authors in \cite{neuralring13} construct a unique ideal associated to a binary code, called the \emph{neural ideal}.

\begin{definition}[\cite{neuralring13}]\label{def:neuralideal}
Let $\C\subseteq \F_2^n$ be a neural code. The \emph{neural ideal of }$\C$, denoted $J_\C$, is the ideal \[
J_\C \od \langle \rho_v \mid v\notin \C \rangle.
\]
Here we adopt the convention that the ideal generated by the empty set is the zero ideal. An ideal in $\F_2[n]$ is called a \emph{neural ideal} if it is equal to $J_\C$ for some code $\C$.
\end{definition}

An ideal in $\F_2[n]$ is a neural ideal if and only if it has a generating set consisting of pseudomonomials \cite{thesis}. This characterization of neural ideals will be useful in proving that certain homomorphisms preserve neural ideals, and in classifying all such homomorphisms. We now formally introduce our objects of interest: the class of homomorphisms which preserve neural ideals.

\begin{definition}\label{def:neuralhomomorphism}
Let $\phi:\F_2[n] \to \F_2[m]$ be a homomorphism of rings. We say that $\phi$ \emph{preserves neural ideals} if the image of any neural ideal in $\F_2[n]$ under $\phi$ is a neural ideal in $\F_2[m]$. That is, for every code $\C\subseteq \F_2^n$ there is a code $\D\subseteq \F_2^m$ so that \[
\phi(J_\C) = J_\D.
\]
\end{definition}

 Because neural ideals in $\F_2[n]$ and codes on $n$ bits are in bijective correspondence we can think of such a map as defining a process for transforming codes on $n$ bits to codes on $m$ bits. Note immediately that composing any two maps which preserve neural ideals yields a map which again preserves neural ideals.

\begin{example}\label{ex:respecting}
Consider the surjective map $\phi:\F_2[6]\to\F_2[3]$ defined by \begin{align*}
x_1\mapsto &\, 0 & x_4\mapsto &\, 1-x_1\\
x_2\mapsto &\, 1-x_3 & x_5\mapsto &\, 1\\
x_3\mapsto &\, 1 & x_6\mapsto &\, x_2\\
\end{align*}
extended algebraically to all of $\F_2[6]$.  Let $J$ be the neural ideal generated by $\{x_1x_2(1-x_3), x_4(1-x_2), x_5x_6(1-x_1)\}$. If we apply $\phi$ to this set of generators then we obtain a set of generators for the ideal $\phi(J)$:\[
\phi(J) = \langle 0, (1-x_1)x_3, x_2\rangle = \langle x_3(1-x_1), x_2\rangle.
\]
Since $\phi(J)$ is generated by pseudomonomials it is a neural ideal. We will see shortly as a result of Theorem \ref{thm:neuralhomomorphism} that the image of \emph{any} neural ideal under $\phi$ is again a neural ideal. 
\end{example}

The following definitions give three useful classes of maps that preserve neural ideals.  For the proof that these maps preserve neural ideals, see Lemma \ref{lem:respectingmaps}. 

\begin{definition}\label{def:permutationmap}
Let $\lambda$ be a permutation of $[n]$. Then the map induced by sending $x_i$ to $x_{\lambda(i)}$ is called a \emph{permutation map} of $\F_2[n]$ and is denoted simply by $\lambda$.
\end{definition}

\begin{definition}\label{def:bitflip}
For any $i\in[n]$ the $i$-th \emph{bit flip} is the map $\delta_i:\F_2[n]\to\F_2[n]$ induced by \[
\delta_i(x_j) = \begin{cases} x_j & j \neq i \\
1- x_j & j = i.
\end{cases}
\] We will call any composition of such maps a \emph{bit flipping map}.
\end{definition}

\begin{definition}\label{def:restrictionmap} 
Let $1\le m\le m' \le n$ and define $\omega_{m,m'}:\F_2[n]\to \F_2[m]$ to be the map induced by \[
\omega_{m,m'}(x_i) = \begin{cases}
x_i & i\le m\\
0 & m< i \le m'\\
1 & m'< i \le n
\end{cases}
\]
Such a map is called a \emph{restriction map} from $\F_2[n]$ to $\F_2[m]$.
\end{definition}


Having described these three classes of neural ideal preserving homomorphisms we can state our first main result, which tells us that every neural ideal preserving homomorphism can be expressed as a composition of these maps.
\begin{theorem}\label{thm:neuralhomomorphism}
Let $\phi:\F_2[n]\to \F_2[m]$ be a homomorphism. Then $\phi$ preserves neural ideals if and only if $\phi$ can be expressed as a composition of the following types of maps:\begin{itemize}
\item[(i)] Bit flipping (Definition \ref{def:bitflip}), 
\item[(ii)] Permutation (Definition \ref{def:permutationmap}), and
\item[(iii)]  Restriction (Definition \ref{def:restrictionmap}).
\end{itemize}
Moreover, we can write $\phi = \omega\circ\lambda\circ\delta$ where $\delta$ is a bit flipping map, $\lambda$ is a permutation, and $\omega$ is a restriction.
\end{theorem}

The proof of Theorem \ref{thm:neuralhomomorphism} can be found in Section \ref{sec:algebraproof}. Theorem \ref{thm:neuralhomomorphism} provides a complete characterization of homomorphisms that preserve neural ideals. By decomposing these homomorphisms into three basic types we gain an understanding of their structure and obtain a compact way to describe any such homomorphism. 

\begin{example}\label{ex:respectingdecomposed}
Recall the map $\phi$ described in Example \ref{ex:respecting}. We can illustrate this map and its decomposition given by Theorem \ref{thm:neuralhomomorphism} by considering the action of $\phi$ on variables. This illustration is provided in Figure \ref{ex:decomposition}, where solid line segments connect a variable to its image (for example $x_1\mapsto 0$ or $x_6\mapsto x_2$) while dashed lines indicate mappings of the form $x_i\mapsto 1-x_j$.
\begin{figure}[h]\[\includegraphics[width=25em]{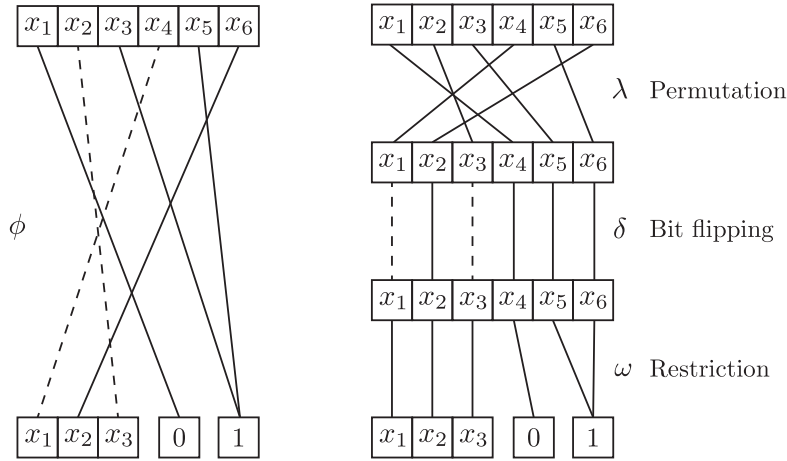}\]
\caption{The homomorphism $\phi$ written as the composition of permutation, bit flipping, and restriction maps. Dashed lines indicate places where $x_i\mapsto 1-x_j$. }\label{ex:decomposition}\end{figure}
\end{example}

We now describe three transformations of codes, which we will see are related naturally to the three types of homomorphisms given in Definitions \ref{def:permutationmap}, \ref{def:bitflip}, and \ref{def:restrictionmap}.

\begin{definition}\label{def:actiononcodes}
Let $\C$ be a code on $n$ bits.\begin{itemize}
\item[(i)] Let $\lambda$ be a permutation of $[n]$. We define the \emph{permutation of $\C$ by $\lambda$} to be the code \[
\lambda(\C) \od \{u\in \F_2^n \mid \supp(u) = \lambda(\supp(c)) \text{ for some } c\in C\}.
\]
\item[(ii)] For any $i\in[n]$ the $i$-th \emph{bit flip of $\C$}, denoted $\delta_i(\C)$, is the code on $n$ bits defined by \[
\delta_i(\C) \od \{u\in \F_2^n \mid \supp(u) = \supp(c) \oplus \{i\} \text{ for some } c\in \C\}.
\]
where $\oplus$ denotes symmetric difference.
\item[(iii)] Let $m$ and $m'$ be integers so that $1\le m \le m' \le n$, and let $\sigma = [n]\setminus [m']$. The \emph{restriction of $\C$ to $(m,m')$} is the code \[
\rest(\C,m,m') \od \{u\in \F_2^m  \mid \supp(u) \cup  \sigma = \supp(c) \text{ for some } c\in \C\}.
\]
\end{itemize}
\end{definition} 

Restricting codes generalizes the notion of taking the link of a face in a simplicial complex. In particular, when $m=m'$ and the supports of codewords in $\C$ forms an abstract simplicial complex $\Delta$, then the supports of vectors in $\rest(\C,m,m')$ form the simplicial complex $\mbox{Lk}_\sigma(\Delta)$ where $\sigma = [n]\setminus [m]$. This connection is important since taking links in the simplicial complex of a code can be used to understand the convexity of that code, as described in \cite{local15}. 

The following theorem allows us to translate precisely between homomorphisms preserving neural ideals and transformations of neural codes.

\begin{theorem}\label{thm:codeeffects} Let $\C$ be a neural code. Then\begin{itemize}
\item[(i)]  $\lambda(J_\C) = J_{\lambda(\C)}$,
\item[(ii)]  $\delta_i(J_\C) = J_{\delta_i(\C)}$, and 
\item[(iii)] $\omega_{m,m'}(J_\C) = J_{\rest(\C,m,m')}$.
\end{itemize}
\end{theorem}

The proof of Theorem \ref{thm:codeeffects} can be found in Section \ref{sec:codeproof}. Finally, we give a geometric interpretation of the behavior of permutation, bit flipping, and restriction maps. We first require a definition.

\begin{definition}\label{def:compatibleregion}
Let $\U = \{U_1,\ldots, U_n\}$ be a collection of sets and let $1\le m\le m' \le n$. The \emph{compatible region} of $(m,m')$ in $\U$ is the set \[
\left(\bigcap_{m'+1\le i \le n}U_i\right)\setminus \left(\bigcup_{m+1\le j\le m'} U_j\right).
\]
\end{definition}
\begin{theorem}\label{thm:geometry}
Let $\C$ be a code with a realization $\U = \{U_1,\ldots, U_n\}$ in a space $X$. 
\begin{enumerate}
\item[(i)]  Let $\lambda$ be a permutation of $[n]$. Then the set $\{U_{\lambda(1)},\ldots,U_{\lambda(n)}\}$ is a realization of $\lambda(\C)$ in $X$. 
\item[(ii)] For any $i$ the collection of sets $\{U_1,\ldots, (X\setminus U_i),\ldots,U_n\}$
is a realization of $\delta_i(\C)$ in $X$. 
\item[(iii)] Let $m$ and $m'$ be integers so that $1\le m\le m' \le n$ and let $X'$ be the compatible region of $(m,m')$ in $\U$. The collection $ \{U_i\cap X' \mid i \in [m] \}$
is a realization of $\rest(\C,m,m')$ in the space $X'$.
\end{enumerate}
\end{theorem}

 The proof of Theorem \ref{thm:geometry} is given in Section \ref{sec:geometryproof}. Together with Theorem \ref{thm:neuralhomomorphism}, Theorem \ref{thm:geometry} allows us to describe any homomorphism preserving neural ideals in terms of operations on the realizations of codes.

We have completely classified all homomorphisms preserving neural ideals and described how they affect the underlying codes. We have also described how these homomorphisms act geometrically on realizations of the associated codes. The correspondence we have described can be summarized as follows.\begin{align*}
\parbox[c][4em][c]{11em}{Permutation maps\\ $\lambda(x_i) =  x_{\lambda(i)}$} && \longleftrightarrow && \parbox[c][4em][c]{10em}{Permuting bits in\\ codewords: $\C\mapsto \lambda(\C)$} && \longleftrightarrow && \parbox[c][4em][c]{10em}{Permuting labels on \\realization: $\U\mapsto \lambda(\U)$}\\
\parbox[c][4em][c]{11em}{Bit flipping maps \\ $\delta_i(x_j) = \begin{cases}1-x_j & j=i; \\ x_j & i\neq j \end{cases}$  } && \longleftrightarrow && \parbox[c][4em][c]{10em}{Flipping $i$-th bit in\\ all codewords of $\C$:\\ $\C\mapsto \delta_i(\C)$} && \longleftrightarrow && \parbox[c][4em][c]{10em}{Replacing $U_i$ with\\
its complement $X\setminus U_i$}\\
\parbox[c][4em][c]{11em}{Restriction maps\\ $\omega_{m,m'}:\F_2[n]\to \F_2[m]$} && \longleftrightarrow && \parbox[c][4em][c]{10em}{Restricting to compatible codewords:\\$\C\mapsto \rest(\C,m,m')$} && \longleftrightarrow && \parbox[c][4em][c]{10em}{Intersecting all sets in realization with the compatible region for $(m,m')$}\\
\end{align*}

The canonical form, an object introduced in \cite{neuralring13},  provides a concise presentation of $J_\C$ and yields information about the underlying code and its realizations. In general it is nontrivial to compute the canonical form of a neural ideal. However, homomorphisms that preserve neural ideals provide a computational shortcut for translating between the related canonical forms.

\begin{definition}\label{def:CF}
Let $J_\C$ be a neural ideal. The \emph{canonical form} of $J_\C$ is the collection of pseudomonomials in $J_\C$ which are minimal with respect to division, and is denoted $\CF(J_\C)$.
\end{definition}

 For any homomorphism $\phi$ preserving neural ideals, the canonical forms  $\CF(J_\C)$ and  $\CF(\phi(J_\C))$ are related as follows.

\begin{theorem}\label{thm:CF}
Let $\phi:\F_2[n]\to \F_2[m]$ be a homomorphism preserving neural ideals. Then $\CF(\phi(J_\C)) \subseteq \phi(\CF(J_\C)) $ for any neural ideal $J_\C$, with equality when $\phi$ is a composition of a permutation and bit flipping map.
\end{theorem}

This result suggests that we can obtain $\CF(\phi(J_\C))$ from $\CF(J_\C)$. In fact we can: if we apply $\phi$ to every pseudomonomial in $\CF(J_\C)$ and select from the result everything which is minimal with respect to division,  the resulting set will be precisely $\CF(\phi(J_\C))$. This result is proven in Section \ref{sec:CFproof}.  

\begin{section}{Classifying homomorphisms that preserve neural ideals}
\label{sec:algebraproof}
In order to establish Theorem \ref{thm:neuralhomomorphism} we first show that permutation, bit flipping, and restriction all preserve neural ideals. To this end, we give a preliminary characterization of homomorphisms that preserve neural ideals.
\begin{lemma}\label{lem:neuralhomomorphism}
Let $\phi:\F_2[n]\to \F_2[m]$ be a homomorphism. Then $\phi$ preserves neural ideals if and only if $\phi$ is surjective and for any pseudomonomial $f\in \F_2[n]$, $\phi(f)$ is either a pseudomonomial or zero.
\end{lemma}
\begin{proof}
Let $\phi$ be a homomorphism which preserves neural ideals. To see that $\phi$ is surjective, note that $\F_2[n]$ is a neural ideal containing 1, so its image must also be a neural ideal containing 1.  Now, let $f\in \F_2[n]$ be a pseudomonomial and recall that $\langle f \rangle$ is a neural ideal,  and $\phi(\langle f\rangle) = \langle \phi(f)\rangle$. Suppose that $\phi(f)$ is not a pseudomonomial. Every factor of a pseudomonomial is again a pseudomonomial, so $\langle \phi(f)\rangle$ does not contain any pseudomonomials. The only neural ideal containing no pseudomonomials is the zero ideal and so $\phi(f) = 0$. 

Conversely, let $J_\C$ be a neural ideal in $\F_2[n]$. Note that since $\phi$ is surjective $\phi(J_\C)$ is an ideal of $\F_2[m]$. Recall that $J_\C = \langle \rho_v \mid v\notin \C\rangle$, and so we will have that \[
\phi(J_\C) = \langle \phi(\rho_v)\mid v\notin \C\rangle.
\]
Removing the $\phi(\rho_v)$ which are zero from the set of generators above we obtain a generating set for $\phi(J_\C)$ consisting only of pseudomonomials. This implies that $\phi(J_\C)$ is a neural ideal and the result follows.
\end{proof}

\begin{lemma}\label{lem:respectingmaps}
Permutation maps, bit flipping maps, and restriction maps all preserve neural ideals.
\end{lemma}
\begin{proof}
Note that permutation maps and bit flipping maps are both automorphisms of $\F_2[n]$, and so are surjective. The image of a pseudomonomial under a permutation map is again a pseudomonomial since we are simply permuting the variables. Under a bit flipping map $\delta_i$ 
any occurrence of $x_i$ in a pseudomonomial will be changed to $(1-x_i)$ and vice versa, while the remaining factors are unchanged, so the image of any pseudomonomial is still a pseudomonomial.
  Thus permutation maps and bit flipping maps satisfy the conditions of Lemma \ref{lem:neuralhomomorphism} and must preserve neural ideals.

Let $\omega_{m,m'}: \F_2[n]\to \F_2[m]$ be a restriction map and $f\in \F_2[n]$ a pseudomonomial. Note that $\omega_{m,m'}(x_i)\in \{0,1,x_i\}$ for all $i$, and $\omega_{m,m'}(1-x_j) \in \{0,1,1-x_j\}$ for all $j$. Hence applying $\omega_{m,m'}$ to $f$ yields either zero (if one of the factors vanishes) or another pseudomonomial (when all factors map to either $1$ or themselves). Furthermore $\omega_{m,m'}$ is surjective since it acts as the identity on $\F_2[m]$ as a subring of $\F_2[n]$. 
Thus by Lemma \ref{lem:neuralhomomorphism} the map $\omega_{m,m'}$ preserves neural ideals.
\end{proof}

Next we describe how a homomorphism that preserves neural ideals affects variables in $\F_2[n]$. This will be critical in decomposing homomorphisms that preserve neural ideals into a composition of the three basic homomorphisms. 

\begin{lemma}\label{lem:lineartolinear}\label{lem:lineartolinearstronger}
Let $\phi:\F_2[n]\to \F_2[m]$ be a map preserving neural ideals. Then\begin{itemize}
\item[(i)]  for each $j\in [m]$ there exists a unique $i\in [n]$ so that $\phi(x_i) \in \{x_j,1-x_j\}$, and
\item[(ii)] for all $i\in[n]$ we have $\phi(x_i) \in \{0,1,x_j,1-x_j\}$.
\end{itemize} 
\end{lemma}
\begin{proof}
The set $\{x_i\mid i\in [n]\}$ generates $\F_2[n]$ as an algebra, and since $\phi$ is surjective the set $\{\phi(x_i)\mid i\in [n]\}$ must generate $\F_2[m]$. This implies that the set \[
\sigma = \{i\in [n]\mid \phi(x_i) \text{ is not constant}\}
\]
has at least $m$ elements. Consequently the pseudomonomial\[
\phi\left(\prod_{i\in \sigma} x_i\right) = \prod_{i\in \sigma} \phi(x_i)  
\]
has degree at least $m$. But no pseudomonomial in $\F_2[m]$ can have degree larger than $m$ and so this pseudomonomial has exactly $m$ linear factors. Each $\phi(x_i)$ must be exactly one of these factors, and so for each $j\in [m]$ we see that there is a unique $i\in[n]$ for which $\phi(x_i) \in \{x_j,1-x_j\}$. The remainder of the $x_i$'s map to constants, but the only constants are $0$ and $1$ and so the second statement of the lemma follows.
\end{proof}

With Lemma \ref{lem:lineartolinearstronger} in hand we prove Theorem \ref{thm:neuralhomomorphism}.

\begin{proof}[Proof of Theorem \ref{thm:neuralhomomorphism}]
 If $\phi$ is a composition of restriction, permutation, and bit flipping maps then by Lemma \ref{lem:respectingmaps} it preserves neural ideals.
To prove the converse, let $\phi:\F_2[n]\to \F_2[m]$ be a homomorphism that preserves neural ideals. We will decompose $\phi$ directly as the composition of a permutation, bit flipping map, and restriction map. First define four sets of indices:\begin{align*}
\alpha_0 &= \{i\in [n] \mid \phi(x_i) = 0\},&\quad
\beta_0 &= \{i\in [n]\mid \phi(x_i) = x_j\},\\
\alpha_1 &= \{i\in [n]\mid \phi(x_i) = 1\}, &\quad
\beta_1 &= \{i\in [n]\mid \phi(x_i) = 1-x_j\}.
\end{align*}
Note by Lemma \ref{lem:lineartolinear} that these four sets completely partition $[n]$. Now define a permutation, bit flipping map, and restriction map as follows:\begin{itemize}
\item Let $\delta:\F_2[n]\to\F_2[n]$ be the bit flipping map which replaces $x_i$ by $1-x_i$ for all $i\in \beta_1$. 
\item Let $\lambda$ to be a permutation of $n$ which maps $\alpha_0$ to the set of indices between $m+1$ and $m+|\alpha_0|$ inclusive, maps $\alpha_1$ to the set of indices between $m+|\alpha_0|+1$ and $n$ inclusive, and maps $i$ to $j$ where $\phi(x_i)\in \{x_j,1-x_j\}$ for all other indices. 
\item Let $\omega:\F_2[n]\to\F_2[m]$ be the restriction map which sends $x_i$ to zero whenever $m+1\le i\le m+|\alpha_0|$ and to one whenever $m+|\alpha_0|+1\le i\le n$. 
\end{itemize} 
We claim that $\phi = \omega\circ \lambda\circ \delta$. To prove this it suffices to show that the image of all variables in $\F_2[n]$ under $\omega\circ \lambda\circ \delta$ is the same as under $\phi$. We consider four cases. \begin{itemize}
\item If $i\in \alpha_0$ then the image of $x_i$ is 0 since $\delta$ does not affect $x_i$, while by construction $\lambda$ maps $x_i$ to an index which is sent to 0 by $\omega$.
\item If $i\in \alpha_1$ then the image of $x_i$ is 1. Again $\delta$ does not affect $x_i$, and $\lambda$ maps $x_i$ to a variable mapped to 1 by $\omega$.
\item If $i\in \beta_0$ then $\delta$ does not affect it. The map $\lambda$ then sends $x_i$ to $\phi(x_i) = x_j$ by construction, and $\omega$ leaves the result unchanged. Thus the image of $x_i$ is exactly $\phi(x_i)$ as desired.
\item Finally, if $i\in \beta_1$ then $\delta(x_i) = 1-x_i$. Applying $\lambda,$ we obtain $1-x_j$ which by construction is $\phi(x_i)$. Applying $\omega$ has no affect on this quantity and so again the image of $x_i$ is exactly $\phi(x_i)$.
\end{itemize}
Since the image under $\omega\circ \lambda\circ \delta$ of each variable is the same as its image under $\phi$ we conclude that $\phi =\omega\circ \lambda\circ \delta$, and the result follows.
\end{proof}

\begin{corollary}\label{cor:idealpmpreimage}
Suppose that $\phi:\F_2[n]\to \F_2[m]$ is a map preserving neural ideals. Then for every neural ideal $J_\D$ in $\F_2[m]$ there exists a neural ideal $J_\C$ in $\F_2[n]$ so that $\phi(J_\C) = J_\D$.
\end{corollary}
\begin{proof}
We first show that every pseudomonomial $f$ has a preimage $\hat{f}$ under $\phi$ such that $\hat{f}$ is a pseudomonomial. By Theorem \ref{thm:neuralhomomorphism} it suffices to prove this for bit flipping maps, permutations, and restrictions. For bit flipping maps and permutations this is clear since they are automorphisms whose inverses also preserve neural ideals. For restrictions, the pseudomonomial itself serves as its own preimage. 

Now let $\{f_1,\ldots, f_k\}$ be a generating set of pseudomonomials for $J_\D$. We can find pseudomonomials $\{\hat{f}_1,\ldots,\hat{f}_k\}$ in $\F_2[n]$ so that $\phi(\hat{f}_i) = f_i$, and letting $J_\C$ be the neural ideal generated by $\{\hat{f}_1,\ldots, \hat{f}_k\}$ we see that $\phi(J_\C) = J_\D$. 
\end{proof}

\end{section}

\begin{section}{The effects of homomorphisms on codes}

\label{sec:codeproof}

Given that we can decompose any homomorphism preserving neural ideals into three basic building blocks (permutation, bit flipping, and restriction) it is natural to ask how these building blocks affect the underlying codes of neural ideals. In particular, if $\phi(J_\C) = J_\D$, how are the codes $\C$ and $\D$ related in terms of the decomposition of $\phi$? Theorem \ref{thm:codeeffects} gives a complete answer to this question. 

\begin{proof}[Proof of Theorem \ref{thm:codeeffects}]
For any $v\in \F_2^n$ define $\lambda(v)$ to be the vector whose support is $\lambda(\supp(v))$. Then note that $\lambda(\rho_v) = \rho_{\lambda(v)}$ for any vector $v$ and permutation $\lambda$. With this we compute directly that \[
\lambda(J_\C) = \langle \rho_{\lambda(v)} \mid v\notin \C\rangle = \langle \rho_v \mid v\notin \lambda(\C)\rangle =  J_{\lambda(\C)}.
\]
For the bit flipping map $\delta_i$ we may define $\delta_i(v)$ to be the vector whose support is $\supp(v)\oplus \{i\}$ so that $\delta_i(\C) = \{\delta_i(c)\mid c\in \C\}$. Then note that $\delta_i(\rho_v) = \rho_{\delta_i(v)}$ and so we can compute \[
\delta_i(J_\C) = \langle \rho_{\delta_i(v)} \mid v\notin \C\rangle = \langle \rho_v \mid v\notin \delta_i(\C)\rangle = J_{\delta_i(\C)}.
\]
Finally, let $\omega_{m,m'}$ be a restriction map and let $\sigma = [n]\setminus [m']$. For simplicity we write $\omega$ for $\omega_{m,m'}$ and let $\D = \rest(\C,m,m')$. We then wish to show that $\omega(J_\C) = J_\D$. 
Recall from Definition \ref{def:actiononcodes} that $u\in \D$ if and only if $\supp(u)\cup \sigma \in \C$.  We will argue that $\omega(J_\C)$ and $J_\D$ contain the same indicator pseudomonomials. 
  Let $\rho_u\in J_\D$ be an indicator pseudomonomial, and let $v$ be the $n$-bit vector whose support is $\supp(u)\cup \sigma$. Note that since $u\notin \D$ we have $v\notin \C$ by the definition of restriction, and hence $\rho_v\in J_\C$. By construction we have $\omega(\rho_v) = \rho_u$, and so $\rho_u\in \omega(J_\C)$.

For the reverse inclusion, observe that since  $\{\rho_v \mid v\notin \C\}$ is a generating set for $J_\C$ the set $\{\omega(\rho_v)\mid v\notin \C\}$ generates $\omega(J_\C)$. All nonzero polynomials in $\{\omega(\rho_v)\mid v\notin \C\}$ will be indicator pseudomonomials in $\F_2[m]$, since applying $\omega$ removes all variables with indices greater than $m$ while leaving those with smaller indices fixed. Each neural ideal is uniquely associated to its generating set of indicator pseudomonomials, and so we conclude that the nonzero elements of $\{\omega(\rho_v)\mid v\notin \C\}$ comprise \emph{all} indicator pseudomonomials in $\omega(J_\C)$.  Thus if $\rho_u\in \omega(J_\C)$ is an indicator pseudomonomial then $\rho_u = \omega(\rho_v)$ for some $\rho_v\in J_\C$. But the only indicator $\rho_v$ whose image under $\omega$ is $\rho_u$ arises from the vector $v$ whose support is $\supp(u)\cup \sigma$. Any other indicator will either disagree on a variable $x_i$ with $i\in[m]$, or it will vanish under $\omega$. Thus if $\rho_u\in \omega(J_\C)$ then $\supp(u)\cup \sigma\notin \C$, which implies that $u\notin \D$ by the definition of restriction. Then $\rho_u\in J_\D$, as desired. Having proven that $\omega(J_\C)$ and $J_\D$ are generated by the same indicator pseudomonomials we conclude that they are equal and the overall result follows.
\end{proof}

\end{section}

\begin{section}{Geometric interpretation of homomorphisms preserving neural ideals}\label{sec:geometry} 

\label{sec:geometryproof}
Finally we give a geometric characterization of permutation maps, bit flipping maps, and restriction maps. Theorem \ref{thm:codeeffects} allows us to translate the behavior of these homomorphisms into the world of neural codes, which we then interpret geometrically in Theorem \ref{thm:geometry}.

\begin{definition}\label{def:codewordregion}
Let $\U = \{U_1,\ldots,U_n\}$ be a collection of sets and $v\in\F_2^n$ be a binary vector. Then the \emph{codeword region} of $v$ in $\U$ is the set\[
A^\U_v \od \left(\bigcap_{v_i=1} U_i\right) \setminus \bigcup_{v_j = 0} U_j.
\]
\end{definition}

The codeword region is the set of points where all the sets $U_i$ for $i\in\supp(v)$ are present, and no others. Note that with this definiton we can write $\C(\U) = \{v\in \F_2^n\mid A^\U_v \neq \emptyset\}$. That is, $\C(\U)$ is the collection of vectors $v$ so that the codeword region $A^\U_v$ is nonempty. This notation is used heavily in the following proof to simplify the translation between a collection of sets and its associated code.

\begin{proof}[Proof of Theorem \ref{thm:geometry}]
Let $\U = \{U_1,\ldots, U_n\}$ be a realization of the code $\C$. The codeword region for $\lambda(v)$ in the realization $\{U_{\lambda(1)}, \ldots, U_{\lambda(n)}\}$ is exactly $A^\U_v$ since we have simply permuted the indices of $\U$. We conclude that the code of the collection $\{U_{\lambda(1)}, \ldots, U_{\lambda(n)}\}$ is $\{\lambda(v)\mid v\in \C\}$, as desired. For bit flipping maps, observe that the codeword region $A^\U_v$ is exactly the codeword region for the vector whose support is $\supp(v)\oplus \{i\}$ in the realization $\{U_1,\ldots, X\setminus U_i ,\ldots, U_n\}$ since taking the set difference with $U_i$ is the same as intersecting with its complement, and intersecting with $U_i$ is the same as taking the difference with its complement. Hence $u$ is in the code of $\{U_1,\ldots, X\setminus U_i ,\ldots, U_n\}$ if and only if $\supp(u)= \supp(c)\oplus \{i\} $ for  some $c\in \C$ and the result follows.

Finally we show that $\{U_i\cap X'\mid i\in[m]\}$ is a realization of $\rest(\C,m,m')$. For convenience let $\V = \{U_i\cap X'\mid i\in[m]\}$ and $\sigma = [n]\setminus [m']$ as in Definiton \ref{def:actiononcodes}. We then consider the region $A^\V_v$ for all vectors $v\in \F_2^m$. If $v = 00\cdots 0$ then the empty intersection in the definition of $A^\V_v$ will be all of $X'$ since we are considering this realization in the space $X'$. The region $A^\V_v$ is nonempty if and only if $X'$ is not covered by $\{U_i\mid i\in [m]\}$, i.e., if and only if $\supp(c) = \sigma$ for $c \in \C$. Hence in this case we have $v\in \C(\V)$ if and only if $\supp(v)\cup\sigma \in \C$. When $v$ is nonzero we compute directly that \begin{align*}
A^\V_v &= \left(\bigcap_{i\in \supp(v)}(U_i\cap X')\right)\setminus \bigcup_{j\in [m]\setminus\supp(v)} (U_j\cap X')\\
&= \left(\bigcap_{i\in\supp(v)\cup \sigma} U_i\right)\setminus \bigcup_{j\in [n]\setminus(\supp(v)\cup\sigma)} U_j.
\end{align*}
In this case we also see that $v\in \C(\V)$ if and only if $\supp(v)\cup\sigma \in \C$. By definition this implies that $\C(\V)=\rest(\C,m,m')$, proving the result.
\end{proof}

This geometric interpretation allows us to specify certain classes of maps which preserve convex neural ideals.

\begin{corollary}\label{cor:preservingconvexity}
Permutation maps preserve convex neural ideals. When $m=m'$ the restriction map $\omega_{m,m'}$ preserves convex neural ideals.
\end{corollary}
\begin{proof}
Relabelling sets in a realization does not affect convexity, and so permutation maps necessarily preserve convex neural ideals. For a restriction map $\omega$ with $m=m'$, we consider the realization given by Theorem \ref{thm:geometry}. If we begin with a convex realization $\U$ of $\C$ then the compatible region $X'$ is simply the intersection of a finite number of convex open sets and hence is itself convex and open. Likewise, all $U_i\cap X'$ are convex and open, so  $\{U_i\cap X'\mid i\in[m]\}$ is a convex realization of $\rest(\C,m,m')$. By Theorem \ref{thm:codeeffects} we have $\omega(J_\C) = J_{\rest(\C,m,m')}$ and so $\omega(J_\C)$  is a convex neural ideal.
\end{proof}

\begin{example}
Figure \ref{fig:geometryexample} shows a code $\C$ on four bits along with a convex realization $\U=\{U_1,U_2,U_3,U_4\}$ in the plane. We can apply Theorem \ref{thm:geometry} to obtain a realization of the code $\rest(\C,2,2)$ in the space $X' = U_3\cap U_4$. As observed in Corollary \ref{cor:preservingconvexity}, this new realization consists of convex sets in a convex space, and hence $\rest(\C,2,2)$ is a convex code.
\end{example}
\begin{figure}[h]
\[\includegraphics[width=28em]{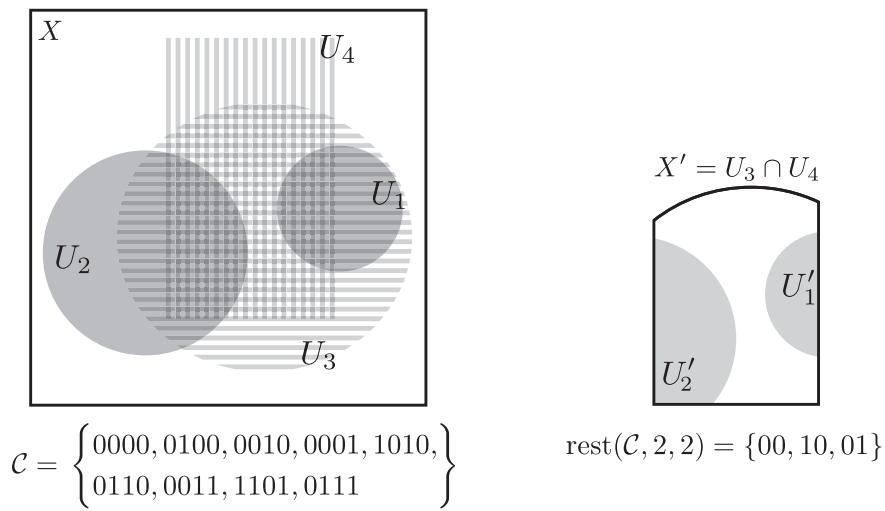}\]
\caption{A convex realization of a code $\C$, and a convex realization of $\rest(\C,2,2)$ as given by Theorem \ref{thm:geometry} and Corollary \ref{cor:preservingconvexity}.}
\label{fig:geometryexample}
\end{figure}
 
Bit flipping maps are the least well behaved among maps preserving neural ideals in terms of respecting convexity of the underlying codes. Indeed, taking the complement of a convex set almost never yields another convex set. Thus in the context of convex codes, bit flipping does not play a natural role. 

It is also important to note that Corollary \ref{cor:preservingconvexity} does not completely describe when convexity is preserved. Corollary \ref{cor:preservingconvexity} gives us a guarantee about when $\phi(J_\C)$ corresponds to a convex code \emph{for all} convex codes $\C\subseteq \F_2^n$, but there may be cases when $J_\C$ and $\phi(J_\C)$ both correspond to convex codes for other homomorphisms $\phi$. For example, it is possible that $\delta_i(J_\C) = J_\C$ for a convex code $\C$. 

\end{section}

\begin{section}{Homomorphisms and the canonical form}

\label{sec:CFproof}

Lastly we prove Theorem \ref{thm:CF}. As a preliminary ingredient we show that every pseudomonomial can be expressed as a sum of indicator pseudomonomials.

\begin{lemma}\label{lem:canopysum}
Let $f$ be a pseudomonomial. Then $f$ can be written as a sum of indicator pseudomonomials.
\end{lemma}
\begin{proof}
We proceed by induction on $n-\deg(f)$. Note that we have $n-\deg(f)\ge 0$ since pseudomonomials never have degree larger than $n$. In the base case that $n-\deg(f) = 0$ we have that $\deg(f) = n$ so $f$ is an indicator and the result follows trivially. If $n-\deg(f) >0$ then there is some variable $x_i$ on which $f$ does not depend. Then we may apply the inductive hypothesis to the pseudomonomials $x_if$ and $(1-x_i)f$. Noticing that $f = x_if + (1-x_i)f$ we obtain the desired result.
\end{proof}

\begin{proof}[Proof of Theorem \ref{thm:CF}]
By Theorem \ref{thm:neuralhomomorphism} it suffices to prove the theorem for permutations, bit flipping maps, and restrictions. In the case of permutations and bit flipping maps the result is clear since these are automorphisms of $\F_2[n]$ and so they provide a bijective correspondence between pseudomonomials that are minimal with respect to division. Thus if $\phi$ is a composition of bit flipping and permutation maps then $\CF(\phi(J_\C)) = \phi(\CF(J_\C))$. 

This leaves the case of a restriction map $\omega$ with parameters $1\le m\le m'\le n$.  In this case let $f\in \CF(\phi(J_\C))$ and write $f = \sum \rho_v$ as guaranteed by Lemma \ref{lem:canopysum}. Recall from the proof of Theorem \ref{thm:codeeffects} that  $\rho_v\in \omega(J_\C)$ if and only if $\rho_u\in J_\C$ where $\supp(u) = \supp(v)\cup \sigma$.  Defining a pseudomonomial \[
h = \prod_{m'+1\le i \le n} x_i \prod_{m+1\le j\le m'} (1-x_j)
\]
we see that it is equivalent to state that $\rho_v\in \omega(J_\C)$ if and only if $h\rho_v\in J_\C$. We conclude that  $\sum h\rho_v$ is in $J_\C$, noting that this is simply the pseudomonomial $hf$. Furthermore we can observe that $\omega(hf) = f$, and so we have shown that there is some pseudomonomial in $J_\C$ which maps to $f$. 

Among all pseudomonomials in $J_\C$ that map to $f$, choose $\hat{f}$ to be minimal with respect to division. We claim that the pseudomonomial $\hat{f}$ is in $\CF(J_\C)$. If $g|\hat{f}$ for a pseudomonomial $g\in J_\C$ then $\omega(g)$ must divide $f$. Since $f\in \CF(\omega(J_\C))$ we see that $\omega(g) = f$. By the minimality of $\hat{f}$ we have $g=\hat{f}$, and so $\hat{f}\in \CF(J_\C)$. We have shown that every pseudomonomial $f$ in  $\CF(\omega(J_\C))$ is also in $\omega(\CF(J_\C))$ and so the result follows. 
\end{proof}
\end{section}

\bibliographystyle{plain}
\bibliography{references}

\end{document}